\documentclass[12pt,reqno]{amsart}
\usepackage{amssymb,amsmath,comment}
\usepackage{color}
\def\({\left(}
\def\){\right)}
\def\Nx{\nabla}
\def\Cal{\mathcal}
\def\eb{\varepsilon}
\def\al{\alpha}
\def\Om{\Omega}
\def\di{\partial_i}
\def\la{\lambda}

\def\divv{\operatorname{div}}

\def\R {\mathbb{R}}

\def\A {{\mathcal A}}

\def\Tr{\operatorname{Tr}}

\def\<{\left<}
\def\>{\right>}

\def \pt {\partial_t}

\def\ext{\operatorname{ext}}
\def \and{\qquad\text{and}\qquad}

\def\Bbb{\mathbb}
\def\Dt{\partial_t}
\def\Dx{\Delta}
\def\rw{\rightarrow}
\def\be{\begin{equation}}
\def\ee{\end{equation}}

\newtheorem{proposition}{Proposition}[section]
\newtheorem{theorem}[proposition]{Theorem}
\newtheorem{corollary}[proposition]{Corollary}
\newtheorem{lemma}[proposition]{Lemma}
\theoremstyle{definition}
\newtheorem{definition}[proposition]{Definition}
\newtheorem{remark}[proposition]{Remark}

\numberwithin{equation}{section}

\def \au {\rm}
\def\t{\theta}

\def\Om{\Omega}
\def \no#1#2#3 {{\bf #1} (#3), #2.}
\def \eds#1#2#3 {#1, #2, #3.}

\title[ 3D Navier--Stokes--Voigt Equations]
{Attractors and their dimensions  for the \\3D
Fractional Navier--Stokes--Voigt Equations}

\author[A. Ilyin, V. Kalantarov and S.Zelik]
{Aleksei  Ilyin${}^{2,6}$, Varga  Kalantarov${}^{3,4}$ and Sergey  Zelik${}^{1,2,5,6}$}
\thanks{ }
\address{${}^1$ Zhejiang Normal University, Department of Mathematics, Zhejiang, China}
\address{${}^2$ Keldysh Institute of Applied Mathematics, Moscow, Russia}
\address{${}^3$Ko{\c c} University, Department of mathematics, Sariyer, Istanbul, Turkiye}
\address{${}^4$ Azerbaijan Technical University, Baku, Azerbaijan}
\address{${}^5$  University of Surrey, Department of Mathematics, Guildford,\\
GU2 7XH,  United Kingdom}
\address{${}^6$ HSE University, Nizhny Novgorod, Russia}


\begin{document}

\begin{abstract}
We study the dimensions of the  attractors for the  fractional
Navier--Sto\-kes--Voigt equations. These equations, which include a
fractional order of the Stokes operator applied to the time
derivative, serve as  natural extensions and regularizations of the
classical Navier--Stokes equations. We give a comprehensive
analysis of the upper bounds for the fractal dimensions of the
attractor in terms of the relevant physical  parameters based on
the advanced spectral inequalities such as Lieb--Thirring and
Cwikel--Lieb--Rosenblum inequalities. These results extend previous
works on the classical Navier--Stokes--Voigt system to the
fractional setting and give an essential improvement of the
estimates known before for the non-fractional case as well.
\end{abstract}

\maketitle
\tableofcontents

 \vskip0.25in

 \vskip0.25in

\section{Introduction }

 \indent We consider the initial-boundary value problem for the three-dimensional fractional Navier–Stokes–Voigt (fNSV) equations in a bounded domain $\Omega \subset \mathbb{R}^3$ with sufficiently smooth boundary $\partial \Omega$:
\begin{equation}\label{0.eqmain}
\begin{cases}
 (1+\alpha^s A^s) \Dt u + (u, \nabla)u - \nu \Delta u+ \nabla p = h,\\
    \divv u = 0,\quad
  u\big|_{\partial \Omega} = 0, \quad   u\big|_{t=0} = u_0.
  \end{cases}
   \end{equation}
Here, $u=u(x,t)=(u_1,u_2,u_3)$ denotes the velocity field,
$p=p(t,x)$ is the pressure, $\nu > 0$ is the kinematic viscosity,
$h\in H^{-1}_\sigma$ is a given external force, $\alpha > 0$ is the
regularization parameter, and $s \in [\frac{1}{2}, 1]$ stands for
the order of the fractional Stokes operator $A$, and
$$
(u,\Nx)u:=\sum_{i=1}^3u_i\partial_{x_i}u,
$$
 see section \ref{s1} for the details.
\par
The Navier--Stokes--Voigt equations (corresponding to $s = 1$ in
\eqref{0.eqmain}) were first proposed by O. Ladyzhenskaya in her
1966 ICM talk \cite{LaC} as a possible regularization of the 3D
Navier--Stokes system. Later, in \cite{Pav}, these equations were
also shown to describe the motion of aqueous polymer solutions. The
system \eqref{0.eqmain} was also introduced in \cite{CLT} as a
regularization tool, particularly useful in numerical simulations
of the Navier--Stokes equations for small values of $\alpha$.
\par
This type of regularization has also been applied in the study of
various nonlinear PDEs, including the nonlinear Klein--Gordon
equation, von K\'{a}rm\'{a}n dynamics, the Kirchhoff equation, and
the 2D surface quasi-geostrophic model (see, e.g., \cite{Lions},
\cite{Chu1}, \cite{KhTi}, and references therein).

The first mathematical analysis of the NSV system under Dirichlet
boundary conditions was carried out by A. Oskolkov in \cite{Osk1},
where global well-posedness was proved. Subsequent studies have
explored the qualitative behavior of solutions; see \cite{Sv},
\cite{Su}, \cite{ZvTu}, and others. In recent years, there has been
growing interest in the NSV equations and various generalizations
of them, driven by their relevance to the challenging 3D
Navier--Stokes problem and their applicability to various physical
models (see, e.g., \cite{Ebr}, \cite{Kra}, \cite{LaPe},
\cite{PuFr}, \cite{Str}).
 For instance, in \cite{KoSv}, the authors studied a NSV system with
 Brinkman--Forchheimer type perturbation  with a cubic
 source term of a ``wrong" sign
$$
\pt u-\al\Dx \pt u-\nu \Dx u-(v, \Nx)u+\Nx p=|u|^2u, \ \divv u=0,
$$
and provided sufficient conditions on the initial data
leading to finite-time blow-up.
\par
The NSV system is also referred to in the literature as the Kelvin--Voigt model
or Oskolkov's system. In \cite{IlZe1}, improved bounds were
derived for the dimension of the global attractor of the NSV semigroup.
In the 2D case, these estimates converge to those for the classical
Navier--Stokes system as $\alpha \to 0$.\\
Another interesting version is  a subdiffusive
Navier--Stokes--Voigt system
\begin{equation}\label{frdif}
\Dt u-\al\Dx \Dt u-\nu \Dx u-(v, \Nx)u+\Nx p=0,\ \ v+\beta  Av=u
\end{equation}
 is introduced in  \cite{Slo} (see also \cite{Kra}) to describe the relaxation process
in polymers.

\par
Finally, we note that regularization terms of the form
$\alpha (-\Delta)^s \partial_t v , s\in (0,1]\ $ have
also been used in the study of the 3D
Navier--Stokes, weakly damped nonlinear wave equations, and Kirchhoff-type
systems (see, e.g., \cite{AlKa}, \cite{Chu1}, \cite{KaZe}, \cite{SaZe}).

\par
The analytic properties of equations of the form \eqref{0.eqmain}
look well-understood and, for instance, the global well-posedness
of solutions can be obtained in a straightforward way
(see e.g. \color{black}\cite{La1,BrLaSa,Ka2,Kalantarov-Levant-Titi,KKO,Osk1}\color{black}). In the 3D case, we need the restriction
$s\ge\frac12$ for this result and for the 2D case $s\ge0$ is enough.
In the case of 3D fractional Euler--Voigt system, which corresponds
to $\nu=0$ in \eqref{0.eqmain}, the assumption $s\ge\frac56$ is needed.
Moreover, some natural convergence results as $\alpha \to 0$
or $\nu,\alpha\to0$ can be also established (see \cite{BrLaSa}).
\par
The study of the long-time behaviour of solutions of \eqref{0.eqmain}
can be initiated using the theory of attractors also in a more or
less straightforward way. For instance, the dissipativity and the
existence of a global attractor $\mathcal A$ for the case $s=1$
is established in \cite{Ka1}.  The finite-dimensionality of this
attractor and some (rough)  upper bounds on its dimension was
obtained in \cite{KaTi1}. These bounds have been improved later,
see \cite{CotiGal} and \cite{IlZe1}. For instance, the following
estimate for the dimension of the attractor for the 3D case with
$s=1$ is obtained in \cite{IlZe1}:
\begin{equation}\label{0.wrong}
\dim_f(\Cal A)\le C G^2(\alpha\lambda_1)^{-\frac34}\min\big\{(\alpha\lambda_1)^{-\frac34},G^2\big\}
\end{equation}
where
$$G=\frac{|\Omega|^{1/2}\|h\|_{L^2}}{\nu^2}$$ is the Grashof number in three dimensions,
$\lambda_1$ is the first eigenvalue of the Stokes operator and $C$ is
some absolute constant. This estimate is valid if $\alpha\lambda_1$
is small enough. The key difference which makes bounding of the
attractor's dimension harder than in the case of Navier--Stokes equation
is the presence of two independent dimensionless parameters:
the first one is the Grashof number $G$ as in Navier--Stokes equation
and the second one
$$G_1=\alpha|\Omega|^{-2/3}\sim \alpha\lambda_1 $$
is new and is specific for the Voigt regularization considered,
so the asymptotics for the dimension depend on the relations
between these two parameters. The Grashof number is assumed big
(since for small $G$ the attractor becomes just one exponentially stable point),
but the parameter $G_1$ may be arbitrary (small, big or intermediate).
\par
The main aim of our paper is to give a comprehensive study of
the dimension of the attractor $\Cal A$ for equation \eqref{0.eqmain}
for various values of parameters $s$, $\alpha$, and $\nu$. In particular,
we study the limit $\alpha\to\infty$ as well,
\color{black}
see \cite{ZIK} for some results in this direction.
\color{black}
 The main result of the paper is the following theorem.
\begin{theorem}\label{Th0.main} Let $\Omega$ be a smooth
bounded domain and $h\in L^2(\Omega)$. Then, for $s\in[\frac34,1]$
the attractor $\Cal A$ of problem \eqref{0.eqmain} possesses the
following upper bound:
\begin{equation}\label{0.est1}
\dim_f(\Cal A)\le C\(\frac{1+G_1^s}{G_1^s}G\)^{\frac32}.
\end{equation}
For the case $s\in[\frac12,\frac34]$ the corresponding estimate reads
\begin{equation}\label{0.est2}
\dim_f(\Cal A)\le C_1\(\frac{1+G_1^s}{G_1^s}\)^{\frac32}G^{6(1-s)}.
\end{equation}
Moreover, for $s=1$ and $G_1^3G^4\ll1$, we have
\begin{equation}\label{0.est3}
\dim_f(\Cal A)\le C_2\(G_1G^3\)^{\frac6{13}}.
\end{equation}
\end{theorem}
We  give the explicit values for the constants involved in these
estimates. We also have to note that the lower bounds for the
attractors dimensions for equation \eqref{0.eqmain} are not known,
so we do not know how good are the estimates in Theorem
\ref{Th0.main}. The only thing which we may say is that, even in
the case $s=1$, the obtained estimates are essentially better than
estimate \eqref{0.wrong} known before. Some exception is the case
$G_1\to\infty$. The limit equation (after the proper time scaling)
reads
 $$
 A^s\Dt u+(u,\Nx)u+\Nx p=\nu\Dx u+h
 $$
 and for the lower bounds for the attractor dimension we
 may use the method suggested in \cite{KISZ}
 (assuming that the key Vishik vortex for this equation exists).
 The realization of this scheme should give us the lower bound of
 order $G$ for the dimension. Thus, we expect a gap between the upper
 bound ($G^{3/2}$) and the lower bound ($G$) even in the case $s=1$ and
 $G_1\to\infty$. We return to this problem somewhere else.
 \par
 The paper is organized as follows. In section \ref{s1}, we briefly
 discuss the functional spaces, which are standard for hydrodynamics
 as well as the fractional powers of the Stokes operator. In section
 \ref{s2}, we discuss the standard facts concerning the well-posedness of
 the problem considered, its dissipativity and the existence of a global attractor.
 \par
 The proof of the main Theorem \ref{Th0.main} is given in section \ref{s3}.
 This proof is heavily based on the advanced spectral inequalities such as
 Berezin--Li--Yau, Lieb--Thirring, Cwikel--Lieb--Rosenblum inequalities
and some combinations of them, some of these combinations have an
independent interest and are presented in the appendix.


\section{Preliminaries: function spaces and  fractional powers of the Stokes operator }\label{s1}
In this section, we briefly recall the standard facts which will be
used throughout of the paper. We start with the spaces of distributions.
\par
 Let $\Omega\subset\R^3$ be a bounded smooth domain in $\R^3$. We
 denote by $D(\Omega)$ and $D_\sigma(\Omega)$ the space $C_0^\infty(\Omega)$
 (resp. $C_{0,\sigma}^\infty(\Omega)$ of divergence free vector fields)
 endowed with the topology of a (locally-convex strict) inductive limit,
 see \cite{Io,Smo}.  Then, the space of distributions $D'(\Omega)$
 (resp. solenoidal distributions $D'_\sigma(\Omega)$) is defined
 as a dual space to $D(\Omega)$ (resp. $D_\sigma(\Omega)$). It is well-known
 that $D(\Omega)$ is dense in $D'(\Omega)$, so for every $l\in D'(\Omega)$
 there is a sequence $l_n\in D(\Omega)$ such that $l_n\to l$  weakly
 in $D'(\Omega)$. We also recall that $l_n\in D(\Omega)$ means that
 there exists a function $g_{l_n}\in D(\Omega)$ such that
$$
l_n\varphi=(g_{l_n},\varphi),
$$
where $(\cdot,\cdot)$ is a standard inner product in $L^2(\Omega)$.
Therefore, the duality between $D(\Omega)$ and $D'(\Omega)$ can be
defined via
\begin{equation}\label{1.dual}
\<l,\varphi\>:=\lim_{n\to\infty}(g_{l_n},\varphi).
\end{equation}
In other words, the space $D'(\Omega)$ is a completion of $L^2(\Omega)$
with respect to the system of seminorms
$\{\|g\|_{\varphi}:=|(g,\varphi)|\}_{\varphi\in D(\Omega)}$
\par
The analogous formulas hold for the solenoidal distributions as well.
To write them out we first introduce the key space
$$
\Cal H:=[D_\sigma(\Omega)]_{L^2(\Omega)}=\{u\in L^2(\Omega)\,: \divv u=0, \ u\cdot n\big|_{\partial\Omega}=0\},
$$
where $[V]_{W}$ is a closure of the set $V$ in the space $W$, see \cite{tem}.
Let $(\cdot,\cdot)_\sigma$ be a restriction of the inner product $(\cdot,\cdot)$
on $L^2(\Omega)$ to the space $\Cal H$.  Then, for every $l\in D'_\sigma(\Omega)$
there is a sequence $l_n\in D_\sigma(\Omega)$ such that $l=\lim_{n\to\infty}l_n$
and the duality between $D'_\sigma(\Omega)$ and $D_\sigma(\Omega)$ is defined via
\begin{equation}\label{1.dual-s}
\<l,\varphi\>_\sigma:=(g_{l_n},\varphi)_\sigma
\end{equation}
and $D'_\sigma(\Omega)$ is a completion of the locally convex space $\Cal H$
generated by the system of seminorms
$\{\|g\|_\varphi:=|(g,\varphi)_\sigma|\}_{\varphi\in D_\sigma(\Omega)}$.
\par
As the next step, we recall that the Sobolev space
$W^{m,p}(\Omega)$, $m\in \mathbb Z_+$, $1\le p\le \infty$, is a
subspace of $D'(\Omega)$ defined by the finiteness of the following norm:
$$
\|u\|_{W^{m,p}}^p:=\sum_{|\beta|\le m}\|D^\beta_x u\|^p_{L^p},
$$
where $\beta:=(\beta_1,\beta_2,\beta_3)\in\Bbb Z^3_+$ and
$D^\beta_x:=\partial^{\beta_1}_{x_1}\partial^{\beta_2}_{x_2}\partial_{x_3}^{\beta_3}$.
As usual,
 $W^{l,p}_0(\Omega)$ is the closure of $D(\Omega)$ in $W^{l,p}(\Omega)$.
 It is convenient to introduce the equivalent truncated norm in the spaces
 $W^{l,p}_0$:
 $$
\|u\|_{W^{l,p}_0}^p:= \sum_{|\beta|= m}\|D^\beta_x u\|^p_{L^p}
 $$
 For the non-integer indexes $m\ge0$ the space $W^{m,p}(\Omega)$ is defined
via the Besov  norms
$$
\|u\|_{W^{m,p}}^p:=\|u\|_{W^{[m],p}}^p+\sum_{|\beta|=[m]}\int_{x\in\Omega}
\int_{y\in\Omega}\frac{|D^\beta_xu(x)-D^\beta_y u(y)|^p}{|x-y|^{n+p\{m\}}}\,dx\,dy.
$$
The spaces $W^{m,p}_\sigma(\Omega)$ and $W^{m,p}_{0,\sigma}(\Omega)$ are defined analogously. In the case $p=2$, we will write the symbol $H$ instead of $W$, see \cite{Ad, Tri} for more details.
\par
Furthermore, we define negative Sobolev spaces $H^{-s}(\Omega)$ (resp. $H^s_\sigma(\Omega)$) as a completion  of $L^2(\Omega)$ (resp. $\Cal H$) with respect to the following norm
\begin{multline*}
\|u\|_{H^{-s}}:=\sup\{|(u,\varphi)|\,:\ \varphi\in D(\Omega),\ \|\varphi\|_{H^s_0}=1\}\ \text{ resp. }\\
\|u\|_{H^{-s}_\sigma}:=\sup\{|(u,\varphi)_\sigma|\,:\ \varphi\in D_\sigma(\Omega),\ \|\varphi\|_{H^s_0}=1\}.
\end{multline*}
We now introduce the Leray--Helmuoltz projector $\Pi$.
On the space $L^2(\Omega)$ it is defined as an orthoprojector
 to the closed subspace $\Cal H\subset L^2(\Omega)$.
In other words, it is defined via bilinear forms for $u\in L^2(\Omega)$,
$\Pi u\in\Cal H$ is defined via the Riesz representation theorem:
\begin{equation}\label{1.pi-h}
(\Pi u,\varphi)_\sigma:=(u,\varphi).
\end{equation}
It is well-known that $\Pi$ can be written in the form of
zero order pseudodifferential operator and, by this reason it
can be extended by continuity to the operator
$\Pi: W^{s,p}(\Omega)\to W^{s,p}(\Omega)$ for all $s\ge0$ and $1<p<\infty$.
Note that, $\Pi$ does not map $W^{s,p}_0(\Omega)$ to $W^{s,p}_{0,\sigma}(\Omega)$
and this causes extra difficulties in the theory of Navier--Stokes
equations in domains.
\par
The Leray--Helmholtz projector for negative Sobolev spaces is
more delicate and  depends strongly on how exactly the duality is
chosen (for simplicity, we will consider the Hilbert case $p=2$ only).
If the above duality between $H^s_{0,\sigma}(\Omega)$ and $H^{-s}_\sigma(\Omega)$
is chosen, the extension of $\Pi$ to negative Sobolev spaces is done via
the following bi-linear form:
\begin{equation}\label{1.triv}
\<\Pi u,\varphi\>_\sigma:=\<u,\varphi\>,\ \ u\in H^{-s}(\Omega),\ \ \varphi\in D_\sigma(\Omega).
\end{equation}
Thus, the ``projector" $\Pi$ acts from $H^{-s}(\Omega)$ to $H^{-s}_{\sigma}(\Omega)$.
\begin{remark}\label{Rem1.fac} Recall that, according to general
theory, the dual to a subspace is a factor space over the annulator
of the subspace in the dual of a larger space. By this reason,
the space $H^{-s}_\sigma(\Omega)$ is a factor space of $H^{s}(\Omega)$
over the space of gradient vector fields $H^{-s}_{grad}(\Omega)$ in
$H^{-s}(\Omega)$. In such interpretation, the map $\Pi$, defined by
\eqref{1.triv},  looks like $\Pi:\,u\to[u]$, where $[u]$ is a class
of equivalence of the element $u\in H^{-s}(\Omega)$. Since $H^{-s}(\Omega)$
is a Hilbert space, any its subspace is complementable, so we may
write a direct sum
$$
H^{-s}(\Omega)=H^{-s}_{grad}(\Omega)\oplus W
$$
for some subspace $W$ such that $W\cap H^{-s}_{grad}=\{0\}$ and try to
identify $H^{-s}_\sigma(\Omega)$ with $W$ and, finally interpret $\Pi$
as a map from $H^{-s}(\Omega)$ to $W$. This would allow us to use the
standard duality \eqref{1.dual} for the space of solenoidal distributions.
The key problem here is that the choice of $W$ should be compatible
with approximations by smooth functions, in particular we should have
the embedding $D_\sigma(\Omega)\subset W$ (and on the smooth functions
the projector $\Pi$ thus defined should coincide with the one, defined
via \eqref{1.pi-h}). This problem is solvable for $s>-\frac12$
(the borderline case $s=-\frac12$ is more delicate and is not discussed here)
and the ``canonical" choice of $W$ is the following one:
\begin{equation}\label{1.W}
H^{-s}_\sigma(\Omega)\equiv W:=\{u\in H^{-s}(\Omega)\,:\ \divv u=0,\ \ (u\cdot n)\big|_{\partial\Omega}=0\},
\end{equation}
see e.g. \cite{Am} for more details. This is closely related with
the uniqueness in the Helmholtz decomposition
\begin{equation}\label{1.helm}
u=v+\Nx p,\ \ u\in H^{-s}(\Omega),\ \ v\in W.
\end{equation}
If $s<-\frac12$, we still have the Helmholtz decomposition of any
vector field to a divergence free and gradient parts (see
\cite{tem} for the details), but this decomposition is defined up
to the harmonic vector field and we cannot fix this harmonic field
since the boundary condition $(v,n)\big|_{\partial\Omega}$ in the
definition of $W$ makes no sense any more. As a result, different
smooth approximations of the same function $u\in H^{-s}(\Omega)$
may have different harmonic components in the Helmholtz
decomposition, so the canonic choice of $W$ does not exist.
Therefore, we are unable to use the standard duality \eqref{1.dual}
for $H^{-s}_\sigma(\Omega)$ if $s<-\frac12$, see also Remark
\ref{Rem1.strange} for more details.
\end{remark}
We are now ready to define the Stokes operator $A: H^1_{0,\sigma}(\Omega)\to H^{-1}_\sigma(\Omega)$. We do this using bi-linear forms, namely,
\begin{equation}\label{1.stokes}
\<Au,v\>_\sigma:=(\Nx u,\Nx v)=\<-\Dx u,v\>,\ \ u,v\in H^1_{0,\sigma}(\Omega).
\end{equation}
Obviously, $A=-\Pi\Dx$. This operator is self-adjoint, positive definite
and realizes an isometry between $H^1_{0,\sigma}(\Omega)$
and $H^{-1}_\sigma(\Omega)$, see e.g., \cite{tem}:
$$
\|Au\|_{H^{-1}_\sigma}=\|u\|_{H^1_{0,\sigma}}.
$$
 Moreover, $A^{-1}:\Cal H\to H^1_{0,\sigma}(\Omega)$ is compact, so by
 the Hilbert--Schmidt theorem,
 there exists a complete orthonormal in $\Cal H$ system of
 eigenvectors $\{e_n\}_{n=1}^\infty$  of $A$ such that
 \begin{equation}
 Ae_n=\lambda_n e_n,\ \ 0<\lambda_1\le\lambda_2\le\cdots,\ \ u=\sum_{n=1}^\infty(u,e_n)e_n, \ \ \|u\|^2_{\Cal H}=\sum_{n=1}^\infty (u,e_n)^2,
 \end{equation}
 where $\{\lambda_n\}_{n=1}^\infty$ are the eigenvalues of the Stokes
 operator $A$.  This diagonalization of $A$ allows us to define a
 scale of Hilbert spaces $H_A^s$, $s\in\R$, which are the completions
 of $\Cal H$ with respect to the following norms
\begin{equation}\label{1.HA}
\|u\|^2_{H^s_A}:=\sum_{n=1}^\infty\lambda_n^s(u,e_n)^2<\infty
\end{equation}
as well as the fractional powers $A^s$ of the operator $A$:
\begin{equation}\label{1.Afrac}
A^s u:=\sum_{n=1}^{\infty}\lambda^s(u,e_n)e_n.
\end{equation}
Obviously, the operators $A^\beta$, $\beta\in\R$, are isometries between $H^s_A$ to $H^{s-2\beta}_A$. We know that $H^0_A=\Cal H$, $H^1_A=H^1_{0,\sigma}$ and $H^{-1}_A=H^{-1}_\sigma$. For smooth domains $\Omega$, we also have the explicit description for all spaces $H^s_A$ via the interpolation
\begin{equation}\label{1.inter}
H^s_A=\{u\in H^s(\Omega)\,: \divv u=0,\ \ (u.n)\big|_{\partial\Omega}=0,\ \ u\big|_{\partial\Omega}=0\}
 \end{equation}
 for all $s\in[0,2]$ with $s\ne\frac12$
 (the trace on the boundary is included when it is well-defined,
 i.e., for $s>\frac12$). For $s=\frac12$, we have a slightly different
formula which includes the part coming from  the Hardy inequality:
 $$
 H^{1/2}_A=\{u\in H^{1/2}(\Omega)\,:\ \int_{\Omega}d^{-1}(x,\Omega)|u(x)|^2\,dx<\infty,\ \ (u.n)\big|_{\partial\Omega}=0\},
 $$
 where $d(x,\Omega)$ is a distance from $x$ to the boundary $\partial\Omega$,
 see e.g. \cite{G}. For the negative $s$ we may use the duality
 arguments, see \cite{Tri}.
 \par
 Thus, since the spaces $H^s_A$ are  close to the usual Sobolev spaces,
 we have the same Sobolev embeddings for them as in the case of the
 classical Sobolev spaces. We will use this fact without further mentioning.
 We also will sometimes use different notation $V^s$ for the spaces $H^s_A$.
 Finally, we mention the alternative notation $D(A^{s/2})$ for $H^s_A$, where
 the space $H^s_A$ is interpreted as the domain of the operator $A^{s/2}$
 considered as an unbounded operator in $\Cal H$.
 \begin{remark}\label{Rem1.strange} We mention one more fact concerning
 the spaces $H^{-1}_\sigma(\Omega)$ and $H^{-1}(\Omega)$ (or more general,
 for the negative Sobolev spaces with $s<-\frac12$), which is presented in
 \cite{GS}. Namely, there exist a sequence of smooth divergence free functions
 $u_n\in D_\sigma(\Omega)$  with the following property:
 \begin{equation}
 \|u_n\|_{H^{-1}_\sigma}\le Cn^{-1}\|u_k\|_{H^{-1}},
 \end{equation}
 in particular, $u_n\to0$ in $H^{-1}_\sigma(\Omega)$ and $u_n\to\Nx\Phi\ne0$, where $\Phi$ is a harmonic function, chosen almost arbitrarily. This shows that, in contrast to what is claimed in many sources, $H^{-s}(\Omega)$ is {\it not embedded} into $H^{-s}_\sigma(\Omega)$ for $s>\frac12$. Moreover, they cannot be realized as subspaces of any Hausdorff topological space. In particular $H^{-1}_\sigma(\Omega)$ is not embedded into $D'(\Omega)$.  The situation here is somehow similar to the case of the space $H^{-1}_N(\Omega):=[H^1(\Omega)]^*$ arising under the study of the Laplacian with Neumann boundary conditions. This space  structurally is a product $H^{-1}(\Omega)\oplus H^{-1/2}(\partial\Omega)$, but there are no canonic splitting of a functional from $H^{-1}_N(\Omega)$ into this sum. We also note that, analogously to the case of the Laplacian, the space $H^s_A$ is not a subspace of $D'_\sigma(\Omega)$ for $s<-\frac32$.
 \end{remark}
 \begin{remark} We note that there are many ways to define the fractional
 Laplacian (see \cite{DL} for a survey) and they generate many ways how to
 define the fractional Stokes operator. The way what we use in the paper is
 often referred as a spectral fractional Stokes operator (we use also the
 so-called regional fractional Laplacian in Appendix to reduce the
 fractional Cwikel--Lieb--Rosenblum inequality to a bounded domain.
 An alternative choice for the problem considered would be the operator
 $\bar A^s:=\Pi(-\Dx)^{s}$, where $(-\Delta)^s$, say, spectral fractional
 Laplacian. Under this choice the alternative fractional Navier--Stokes--Voigt
 equation reads
 \begin{equation}\label{1.eq-alt}
 (1+\alpha^s(-\Dx)^s\Dt u+(u,\Nx)u+\Nx p=\Dx u+h,\ \divv u=0,\ \ u\big|_{t=0}=u_0,\ u\big|_{\partial\Omega}=0.
 \end{equation}
 Unfortunately, in general, $\bar A^s\ne A^s$, so this equation
 {\it differs} from the one, considered in our paper. Our choice motivated
 by the fact that $A^s$ and $A$ have a common base of eigenvectors, in
 contrast to the operators $\bar A^s$ and $\bar A$ and this  simplifies
 the technique related with the upper bounds for the dimension of the
 attractor, which we use. However, we believe that this difference is
 not crucial and the technique is applicable to equation \eqref{1.eq-alt}
 up to some extra technicalities.
 \end{remark}

 \section{Well-posedness, dissipativity,  and attractors}\label{s2}
In this section, we start to consider the fractional Navier--Stokes--Voigt
system. Since the results of this section looks known and the technique is
standard, we skip many details (except of the proof of the energy equality)
and give only a brief exposition. As usual, we first identify what is a
weak solution of our equation.
\begin{definition}\label{2.sol} Let $s\in[\frac12,1]$ and let
$h\in H^{-1}_\sigma(\Omega)$ and $u_0\in\Cal H$. Then
\begin{equation}\label{2.reg}
u\in C_w([0,T], V^s)\cap L^2((0,T),V^1)
\end{equation}
(where $u\in C_w([0,T],V^s)$ means that $u(t)$ is continuous in time
in the weak topology of $V^s$, so the boundary condition is well-defined)
is a weak solution of \eqref{0.eqmain} if, for every test function
$\varphi\in C_0^\infty([0,T],D_\sigma(\Omega))$, we have the identity
\begin{multline}\label{2.dist}
-\int_0^T\<(1+\alpha^sA^s)u(t),\Dt \varphi(t)\>_\sigma\,dt+
\nu\int_0^T(\Nx
u(t),\Nx\varphi(t))\,dt+\\+\int_0^T((u(t),\Nx)u(t),\varphi(t))\,dt=
\int_0^T\<h,\varphi(t)\>_\sigma\,dt.
\end{multline}
The first term in this equality is well-defined since \color{black}
$u(t)\in V^s$ implies $A^su\in V^{-s}\subset D'_\sigma(\Omega)$
($s\in[\frac12,1]$ and, in particular, $s<\frac32$).

\color{black} The term with nonlinearity is also well-defined since
$(u,\Nx)u\in L^1(0,T; L^{\color{black}3/2\color{black}}(\Omega))$ and the rest terms are obvious.
\end{definition}
 Our next step is to check the energy equality.
 \begin{lemma}\label{Lem2.en} Let $u(t)$ be a weak solution of \eqref{0.eqmain}. Then  the function
 $$
 t\to \|u(t)\|_{L^2}^2+\alpha^s\|u(t)\|_{V^s}^2
 $$
 is absolutely continuous and the following energy identity holds:
 \begin{equation}\label{2.en}
 \frac12\frac d{dt}\(\|u(t)\|^2_{L^2}+\alpha^s\|u(t)\|_{V^s}\)
 +\nu\|\Nx u(t)\|^2_{L^2}=\<h,u(t)\>_\sigma
 \end{equation}
 almost everywhere in time.
 \end{lemma}
 \begin{proof}To get the result, we need to obtain some estimates for the inertial term. Indeed, using the embeddings $V^s \subset L^{6/(3-2s)}$ and $V^1\in L^6$ together with the regularity \eqref{2.reg} and the H\"older inequality, we arrive at
 $$
 (u,\Nx)u\in L^2([0,T],L^{\color{black}\frac3{3-s}\color{black}}).
 $$
 We now use that $s\ge\frac12$. Then using that $L^q\subset V^{-\theta}$
 if $\frac1q=\frac12+\frac{\theta}3$, we finally get
 $$
 (u,\Nx)u\in L^2((0,T),V^{\color{black}s-\frac32\color{black}})
 \subset L^2((0,T),V^{-1}).
 $$
 The rest is standard. From \eqref{2.dist}, we conclude by
 continuity that the distributional derivative
 $\Dt(1+\alpha^sA^s)u$ can be extended by continuity to the
 linear functional on $L^2((0,T),V^{-1})$ and consequently
 $$
 (1+\alpha^sA^s)\Dt u\in L^2((0,T),V^{-1}) \
 \text{ and }\ (1+\alpha^sA^s)^{1/2}\Dt u\in V^{-1+s}.
 $$
 Moreover, from \eqref{2.reg}, we know that
 $(1+\alpha^sA^s)^{1/2}u(t)\in L^2(V^{1-s})$ and \cite[Lemma III.1.2]{tem}
  gives us the desired absolute continuity as well as the equality
 \begin{multline}\label{2.tem}
 \frac12\frac d{dt}\(\|u(t)\|_{L^2}^2+\alpha^s\|u(t)\|^2_{V^s}\)=\<(1+\alpha^sA^s)^{1/2}\Dt u(t),(1+\alpha^sA^s)^{1/2}u(t)\>_\sigma=\\=\<(1+\alpha^sA^s)\Dt u(t),u(t)\>_\sigma
 \end{multline}
 Moreover,  the first term in  \eqref{2.dist} can be rewritten
 as $\int_0^T\<(1+\alpha^sA^s)\Dt u(t),\varphi(t)\>_\sigma\,dt$
 and after that the class of test functions can be extended by
 continuity to $L^2((0,T), V^{1})$ which contains the solution
 $u(t)$. Therefore, using this modified identity \eqref{2.dist}
 with $\varphi=u$ and using \eqref{2.tem}, we have
 $$
 \frac12\(\|u(t)\|^2_{L^2}+\alpha^s\|u(t)\|^2_{V^s}\)\big|_{t=0}^{t=T}+
 \nu\int_0^T\|\Nx u(t)\|^2_{L^2}dt=\int_0^T\<h,u(t)\>_\sigma\, dt.
 $$
 Finally, replacing the time interval $[0,T]$ by $[\tau,T]$ for any $\tau\in[0,T)$, we get the analogue of the last equality with the integration limits $\tau$ and $t$ and this is nothing more than the equivalent integral form of the energy identity. This finishes the proof of the lemma.
 \end{proof}
 The standard arguments based on the energy equality allow to verify that any weak solution $u(t)$ is strongly continuous in time
 $$
 u\in C([0,T],V^s).
 $$
 Moreover, due to the regularity estimates of  Lemma \ref{Lem2.en}, the
 equality \eqref{2.dist} can be rewritten in the following equivalent form
 \begin{equation}\label{2.st}
 (1+\alpha^sA^s)\Dt u+Au+\Pi(u,\Nx)u=h
 \end{equation}
 and the equation is understood as the equality in the space
 $L^2(0,T;H^{-1}_\sigma(\Omega))$.
\par
The next theorem gives the well-posedness and dissipativity for
equation \eqref{0.eqmain}
\begin{theorem} \label{Th2.main1} Let $s\in[\frac12,1]$ and
$h\in H^{-1}_\sigma(\Omega)$. Then, for any $u_0\in V^s$, there is a
unique weak solution $u(t)$ of problem \eqref{0.eqmain} defined for
all $t\in\R_+$. Moreover, this solution satisfies the following
dissipative estimate:
 \begin{equation}\label{2.dis}
 \|u(t)\|_{V^s}^2+\int_0^te^{-\beta(t-s)}\|\Nx u(s)\|^2_{L^2}\,ds\le C(\|u_0\|^2_{V^s}e^{-\beta t}+\|h\|_{V^{-1}}^2),
 \end{equation}
 where the positive constants $\beta$ and $C$ are independent of $t$ and $u_0$.
 \end{theorem}
 \begin{proof} Let us first prove uniqueness. Indeed, let $u_1,u_2$
 be two weak solutions and let  $v:=u_1-u_2$. Then, this function solves
 \begin{equation}
 (1+\alpha^sA^s)\Dt v+\nu Av+\Pi (u_1,\Nx)v+\Pi(v,\Nx)u_2=0.
 \end{equation}
 Writing down the energy identity for this equation
  (testing it with $v(t)$ in $V^{-1}$), we get
 \begin{equation}
 \frac12\frac d{dt}\(\|v(t)\|^2+\alpha^s\|v(t)\|^2_{V^s}\)+
 \nu\|\Nx v(t)\|^2_{L^2}+((v(t),\Nx) u_2(t),v(t))=0.
 \end{equation}
 Using the Poincar\'e inequality together with the
 Young inequality we conclude from here that
 $$
 \frac d{dt}\(\|v(t)\|^2_{L^2}+\alpha^s\|v(t)\|^2_{V^s}\)+
 \beta\(\|v(t)\|^2+\alpha^s\|v(t)\|^2_{V^s}\)
 \le C|((v(t),\Nx)u_2(t),v(t))|
 $$
  for some positive constants $C$ and $\beta$. Estimating the  term
  on the right-hand side as follows
  $$
  |((v,\Nx)u_2,v)|\le \|v\|_{L^3}\|v\|_{L^6}\|\Nx u_2\|_{L^2}\le \alpha^s\beta\|\Nx v\|^2_{L^2}+C\|\Nx u_2\|^2_{L^2}\|v\|^2_{V^{1/2}},
  $$
  where we have used the embeddings $V^1\subset L^6$ and $V^{1/2}\subset L^3$,
  and using that $V^{1/2}\subset V^s$ (since $s\ge\frac12$), we finally get
  $$
  \frac d{dt}\(\|v(t)\|^2_{L^2}+\alpha^s\|v(t)\|^2_{V^s}\)\le C\|\Nx u_2(t)\|^2_{L^2}\(\|v(t)\|^2_{L^2}+\alpha^s\|v(t)\|^2_{V^s}\)
  $$
  Integrating this inequality and using that $\|\Nx u_2(t)\|_{L^2}\in L^1(0,T)$, we end up with
  $$
  \(\|v(t)\|^2_{L^2}+\alpha^s\|v(t)\|^2_{V^s}\)\le C_T\(\|v(0)\|^2_{L^2}+\alpha^s\|v(0)\|^2_{V^s}\)
  $$
  for all $t\le T$. This gives uniqueness and even the Lipschitz
  continuity of the solution with respect to initial data.
  \par
  The existence of a solution can be proved in a standard way
  using  Galerkin approximations based on a priori estimate \eqref{2.dis},
 so we leave the proof to the reader and derive a priori estimate \eqref{2.dis} only.
  \par
  Let $u(t)$ be a solution of \eqref{0.eqmain}. Then, using the Poincar\'{e} and Young inequalities, we derive from the energy equality \eqref{2.en} that
  $$
  \frac d{dt}\(\|u(t)\|^2_{L^2}+\alpha^s\|u(t)\|^2_{V^s}\)+
  \beta\(\|u(t)\|^2_{L^2}+\alpha^s\|u(t)\|^2_{V^s}\)+\beta\|\Nx u(t)\|^2_{L^2}\le C\|h\|^2_{V^{-1}}
  $$
for some positive constants $\beta$ and $C$. Integrating this estimate, we get \eqref{2.en} and finish the proof of the theorem.
 \end{proof}
Thus, we have proved that, under the assumptions of Theorem \ref{Th2.main1} problem \eqref{0.eqmain} generates a locally Lipschitz continuous dissipative semigroup in the phase space $V^s$:
\begin{equation}\label{2.sem}
S(t):V^s\to V^s,\ \ S(t)u_0:=S(t).
\end{equation}
The next Corollary is crucial for our study of the upper bounds for the dimension of the attractor.
\begin{corollary}\label{Cor2.est} Let the assumptions of Theorem \ref{Th2.main1} hold and let $u(t)$ be a solution of problem \eqref{0.eqmain}. Then, the following inequality is valid
\begin{equation}\label{2.mean}
\limsup_{T\to\infty}\frac1T\int_0^T\|\Nx u(t)\|^2_{L^2}\,dt\le \frac{\|h\|^2_{V^{-1}}}{\nu^2}.
\end{equation}
For the case when $h\in L^2(\Omega)$, the right-hand side should be replaced by $\frac{\|h\|^2_{L^2}}{\nu^2\lambda_1}$ where $\lambda_1$ is the first eigenvalue of the Stokes operator.
\end{corollary}
\begin{proof} Using that
$$
\<h,u\>_\sigma\le\|h\|_{V^{-1}}\|u\|_{V^1}\le\frac1{2\nu}\|h\|^2_{V^{-1}}+\frac\nu2\|\Nx v\|^2_{L^2}
$$
we derive from energy inequality \eqref{2.en} that
$$
\frac d{dt}\(\|v(t)\|^2_{L^2}+\alpha^s\|u(s)\|_{V^s}^2\)+\frac\nu2\|\Nx u(t)\|^2_{L^2}\le \frac{\|h\|^2_{V^{-1}}}{2\nu}.
$$
Integrating this inequality with respect to $t\in[0,T]$, we arrive at
$$
\frac\nu{2T}\int_0^T\|\Nx u(t)\|^2_{L^2}\,dt\le \frac1{2\nu}\|h\|^2_{L^2}+\frac1T\(\|u(0)\|^2_{L^2}+\alpha^s\|u(0)\|_{V^s}^2\).
$$
Since the second term in the right-hand side tends to
zero as $T\to\infty$, this estimate gives the desired estimate
\eqref{2.mean}. If $h\in L^2$, the Poincar\'{e} inequality gives
$\|h\|_{V^{-1}}^2\le\lambda_1^{-1}\|h\|^2_{L^2}$ and the corollary is proved.
\end{proof}
We now turn to attractors, We recall that a set $\Cal A$
is a (global) attractor for the semigroup $S(t)$ acting in a Banach space $V^s$ if
\begin{itemize}
\item{ $\Cal A$ is a compact subset of $V^s$.}
\item{$\Cal A$ is strictly invariant, i.e., $S(t)\Cal A=\Cal A$ for all $t\ge0$.}
\item{$\Cal A$ attracts the images of all bounded sets of $V^s$ as $t\to\infty$, i.e., for every bounded set $B\subset V^s$ and every neighbourhood $\Cal O(\Cal A)$ of the set $\Cal A$, there exists $T=T(B,\Cal O)$ such that
    $$
    S(t)B\subset\Cal O(\Cal A)
    $$
    for all $t\ge T$.}
    \end{itemize}
The main result of this section is the following theorem.
\begin{theorem}\label{Th2.main2} Let the assumptions of Theorem \ref{Th2.main1} hold. Then the solution semigroup $S(t)$ associated with problem \eqref{0.eqmain} possesses a global attractor $\Cal A$ in the phase space $V^s$ which has the following structure:
$$
\Cal A=\Cal K\big|_{t=0},
$$
where $\Cal K\subset C_b(\R,V^s)$ is the set of all bounded complete (i.e., defined for all $t\in\R$) solutions of equation \eqref{0.eqmain}. We refer to this attractor as an attractor of equation \eqref{0.eqmain}.
\end{theorem}
\begin{proof}[Idea of the proof] According to
the general theory of attractors, see \cite{BaVi,Tem,Zelik1}
and references therein, we need to verify the following properties
(a) dissipativity, (b) continuity of $S(t)$ with respect to the
initial data, (c) asymptotic compactness. The property (c) means
that every sequence
$\{S(t_n)u_{0,n}\}_{n=1}^\infty$, such that $u_{0,n}$ is bounded
in $V^s$ and $t_n\to\infty$, is precompact in $V^s$.
\par
In our case property (a) follows from the dissipative estimate
\eqref{2.dis} and property (b) is actually verified in the
uniqueness part proof of Theorem \ref{Th2.main1}, so we only
need to verify asymptotic compactness.
\par
We use the so-called energy method to do so, see
\cite{ball,Tem,Rosa,Zelik1} and references therein for more
details. Namely, we first extract a subsequence
(which we also denote by $S(t_n)u_{0,n}$ which converges weakly
to some $\xi\in\Cal A_w$, where $\Cal A_w$ is the attractor of
$S(t)$ in a weak topology of $V^s$ (it is possible to do due to the
Banach--Alaoglu theorem since the sequence is bounded in the reflexive space $V^s$).
After that we pass to the limit in some integrated versions of energy equality
\eqref{2.en} and establish the convergence of the norms $\|S(t_n)u_{0,n}\|_{V^s}$
to $\|\xi\|_{V^s}$ which gives us the desired asymptotic compactness.
The structure formula is also a corollary of the above mentioned theorem
on the attractor's existence and the theorem is proved.
\end{proof}
\begin{remark} It is not difficult to prove that the attractor $\Cal A$
is more regular if the function $h$ is. In particular, it is $C^\infty$-smooth
if $h$ is $C^\infty$. The existence of an attractor can be proved by many ways,
in particular, establishing the compactness of the semigroup $S(t)$
(in the case $s<1$ or splitting $S(t)$ on a decaying and compact components
when $s=1$). Our choice of the energy method is motivated by two reasons:
(1) it allows to treat the cases of $s<1$ and $s=1$ in a unified way and
(2) it does not require further regularity of the function $h$.
\end{remark}
\begin{remark}\label{Rem-saba} It follows from Corollary \ref{Cor2.est} that the following limit exists and bounded from above by
\begin{equation}\label{2.lim}
\lim_{t\to\infty}\frac1t\sup_{u\in\Cal K}\int_0^t\|\Nx u(s)\|^2_{L^2}\,ds\le \frac{\|h\|^2_{V^{-1}}}{\nu^2}.
\end{equation}
Indeed, we claim that the function $t\to Q(t):=\sup_{u\in\Cal K}\int_0^t\|\Nx u(s)\|^2_{L^2}\,ds$ is subadditive. Indeed,
\begin{multline*}
Q(t+\tau)=\sup_{u\in\Cal K}\bigg\{\int_0^t\|\Nx u(s)\|^2_{L^2}\,ds+\int_t^{t+\tau}\|\Nx u(s)\|^2_{L^2}\,ds\bigg\}\le\\\le \sup_{u\in\Cal K}\int_0^t\|\Nx u(s)\|^2_{L^2}\,ds+\sup_{u\in\Cal K}\int_0^\tau\|\Nx u(t+s)\|^2_{L^2}\,ds\le Q(t)+Q(\tau),
\end{multline*}
where we have implicitly used the invariance of $\Cal K$. Thus, the existence of the limit is proved and the estimates follows from Corollary \ref{Cor2.est}.
\end{remark}

\section{Upper bounds for the dimension of the  attractor}\label{s3}
In this section, we obtain a number of estimates for the
dimension of the attractor and study their dependence on
the exponent $s$ and the dimensionless parameters $G$ and $G_1$.
\par
We start with reminding the definition of a fractal (box-counting, Minkowski)
dimension.
\begin{definition} Let $\Cal A$ be a compact set in a metric space $V^s$.
Then, by the Hausdorff criterion, for any $\eb>0$, $\Cal A$ can be covered
by finitely many $\eb$-balls. Let $N_\eb(\Cal A,V^s)$ be the minimal number
of such balls. Then the (upper) fractal dimension of $\Cal A$ is defined as
follows:
$$
\dim_f(\Cal A, V^s):=\limsup_{\eb\to0}\frac{\log_2N_\eb(\Cal A,V^s)}{\log_2\frac1\eb}.
$$
\end{definition}
 We will use the volume contraction method combined with various spectral inequalities collected in Appendix A. Recall that the key tool for the volume contraction method is the following theorem, see \cite{blin,Ch-I,hunt,Tem}.
\begin{theorem}Let $H$ be a Hilbert space and let a semigroup $S(t): H\to H$ possesses a global attractor $\Cal A$ in $H$. Assume also that $S(t)$ is uniformly (quasi)differentiable with respect to the inial data on the attractor and that the corresponding quasi-differential is continuous on it. Then, the fractal dimension of the attractor is less than $n$ if $S(t)$ contracts infinitesimal $n$-dimensional volumes on $\Cal A$.
\end{theorem}
In our case, the semigroup $S(t)$, which corresponds to the fractional
Navier--Stokes--Voigt equation is $C^\infty$-smooth with respect
to initial data (we have verified above the Lipschitz continuity of it,
the estimates for the derivatives are analogous and we leave them to
the reader), so the differentiability assumptions of the theorem
are satisfied and we only need to concentrate on the estimation of
the volume contraction factor, which is defined as follows
$$
\omega_N(S(\cdot)):=\lim_{T\to\infty}
\sup_{u_0\in\mathcal A}\|\Lambda^N S'(T)(u_0)\|_{\Lambda^n H}^{1/T},
$$
where $\Lambda^N$ means $N$th exterior power, see \cite{Tem} for more details.
\par
Very often the derivative $v(t):=S(t)(u_0)\xi_0$ of the operators
$S(t)$ solves the linear equation
\begin{equation}\label{3.var}
\frac d{dt}v(t)=\Cal L_{u(t)}v(t),\ \ v\big|_{t=0}=\xi_0,\ \ u(t):=S(t){u_0}
\end{equation}
for some linear  operator $\Cal L_{u(t)}$. Then, under some mild
assumptions, for instance, the validity of the energy equality
$$
\frac12\frac d{dt}\|v(t)\|_H^2=(\Cal L_{u(t)}v(t),v(t))\in L^1(0,T)
$$
in the sense of distributions would be enough, the analogue of the Liouville
theorem holds, see \cite{Tem}, which gives the inequality
$$
\omega_N(S)\le\exp\( \lim_{T\to\infty}\frac1T\sup_{u_0\in\mathcal A}\int_0^T\Tr_N\Cal L_{u(t)}\,dt\),
$$
where
\begin{equation}\label{3.Tr}
\Tr_N\Cal L_{u(t)}:=\sup\biggl\{\sum_{k=1}^N(\Cal L_{u(t)}\theta_k,\theta_k)\,:
\,
(\theta_i,\theta_j)=\delta_{i,j},
\ i,j=1,\dots, N\biggr\}
\end{equation}
Thus, the semigroup $S(t)$ will contract $N$-dimensional volumes if
\begin{equation}\label{3.q}
q(N):= \lim_{T\to\infty}\frac1T\sup_{u_0\in\mathcal A}\int_0^T\Tr_N\Cal L_{u(t)}\,dt<0
\end{equation}
(the existence of the limit can be proved exactly as in Remark \ref{Rem-saba})
and, therefore, for estimating the fractal dimension dimension of the attractor,
we only need to find as small as possible $N$ satisfying \eqref{3.q}.
\par
We now return to the fractional Navier--Stokes--Voigt equation. It is
convenient for us to scale time $t\to\alpha^{-s}t$, introduce $\eb:=\alpha^{-s}$
and rewite our equation in the form of
\begin{equation}\label{3.scaled}
(\eb+A^s)\Dt u+\Pi((u,\Nx)u)+\nu Au=h.
\end{equation}
Obviosly, the attractor does not change after this scaling.
 Our first theorem gives an estimate for the dimension of the attractor,
 which is reasonable to use when $\alpha$ is not very small.

\begin{theorem}\label{Th3.main1} Let $h\in L^2(\Omega)$ and $s\in[\frac12,1]$ Then the fractal dimension of the attractor $\Cal A$ of the fractional Navier--Stokes--Voigt equation possesses the following estimate:
\begin{equation}\label{3.est1}
\color{black}
\dim_f(\Cal A, V^s)\le \(\frac52C_{Sob,6}^3C_{BLY}^{-3/2-s}\)^3
\(\frac{1+C_{BLY}^sG_1^s}{G_1^s}\)^3G^6,
\color{black}
\end{equation}
where $V^s=H^s_A:=D(A^{s/2})$, $G=|\Omega|^{1/2}\frac{\|h\|_{L^2}}{\nu^2}$ is the Grashof number, $G_1=\alpha|\Omega|^{-2/3}$ is the second dimensionless number and the constants are defined in Appendix.
\par
Moreover, if $s\in[\frac34,1]$, we have
\begin{equation}\label{3.est2}
\dim_f(\Cal A,V^s)\le
\biggl(\frac{\color{black}5\color{black}}3\sqrt{\frac23}
C_{s,CLR}
\color{black}C_{BLY}^{-\frac32}\color{black}\biggr)^{3/2}
\(\frac{1+C_{BLY}^sG_1^s}{G_1^s}\)^{3/2}G^{3/2}.
\end{equation}
Finally, for $s\in[\frac12,\frac34]$, we have
\begin{equation}\label{3.est-mid}
\dim_f(\Cal A, V^s)\le \biggl(\frac{\color{black}5\color{black}}3\frac{2^{2(1-s)}}{3^{2(1-s)}}
C_{s,CLR}C_{BLY}^{\color{black}2s-3\color{black}}
\color{black}C_{Sob,6}^{6-8s}\color{black}\biggr)^{3/2}\(\frac{1+C_{BLY}^sG_1^s}{G_1^s}\)^{3/2}G^{6(1-s)}.
\end{equation}
\end{theorem}
\begin{proof} In our case the semigroup $S(t)$, which corresponds
to fractional Navier--Stokes--Voigt equation \eqref{3.scaled}, acts
in the  Hilbert space $V^s$ where we introduce the equivalent inner
product
\begin{equation}\label{3.norma}
(u_1,u_2)_{V^s_\alpha}:=\eb (u_1,u_2)+(A^su_1,u_2)=((\eb+A^s)u_1,u_2).
\end{equation}
The equation of variations  reads
 \begin{equation}\label{3.var1}
 (\eb+A^s)\Dt v=-\nu A-\Pi\((u,\Nx)v+(v,\Nx u)\)
 \end{equation}
 and, therefore, $\Cal L_{u(t)}v=(\eb+A^s)^{-1}\(-\nu A-\Pi\((u,\Nx)v+(v,\Nx u)\)\)$. Following the general theory, we only need to estimate the $N$-dimensional trace of $\Cal L_{u(t)}$ in the space $V^s$ with inner product \eqref{3.norma}. Let $\{\theta_k\}_{k=1}^N$ be an orthonormal system
 \color{black} with respect to \eqref{3.norma}\color{black}.
Then,
\begin{multline}\label{3.base}
(\Cal L_{u(t)}\theta_k,\theta_k)_{V^s_\alpha}= -\nu\|\Nx
\theta_k\|^2_{L^2}-((\theta_k,\Nx) u(t),\theta_k)
+((u,\Nx)\theta_k,\theta_k))\le\\\le -\nu\|\Nx\theta_k\|^2_{L^2}+
|((\theta_k,\Nx)u,\theta_k)|.
\end{multline}
The RHS of this inequality can be estimated in several
different ways. The first way is purely based on the min-max inequality
(and does not utilize more sophisticated spectral inequalities), namely, we use that

\color{black}

  \begin{multline*}
  |((\theta_k,\Nx)u_k,\theta_k)|=|((\theta_k,\Nx)\theta_k,u)|\le
\|\theta_k\|_{L^3}\|\nabla\theta_k\|_{L^2}\|u\|_{L^6}\le\\
\le\frac\nu 2\|\Nx\theta_k\|^2_{L^2}+
  \frac1{2\nu}
  \|u\|_{L^6}^2\|\theta_k\|^2_{L^3}\le \frac\nu 2\|\Nx\theta_k\|^2_{L^2}+
  \frac{C_{Sob,6}^2}{2\nu}\|\Nx u\|^2_{L^2}\|\theta_k\|^2_{L^3},
  \end{multline*}
and, therefore, using the min-max principle and \eqref{A.mm}, we end up with
\begin{multline}
\Tr\Cal L_{u(t)}\le -\frac\nu2\sum_{k=1}^N\|\Nx\theta_k\|^2_{L^2}+
\frac{C_{Sob,6}^2}{2\nu}\|\Nx u(t)\|^2_{L^2}\sum_{k=1}^N\|\theta_k\|^2_{L^3}\le\\
\le -\frac\nu2\hat\zeta(\alpha,s,N)+\frac34
\frac{\|\Nx u(t)\|_{L^2}^2}{\nu}C_{Sob,6}^3
|\Omega|^{\frac23(s-\frac12)}C_{BLY}^{\frac12-s}N^{\frac23(2-s)}.
\end{multline}
Together with \eqref{A.gest1}, and averaging in time
$$
\limsup_{t\to\infty}\frac1t\int_0^t\|\Nx u(s)\|^2_{L^2}\, ds\le
\frac{\|h\|^2_{L^2}}{\nu^2\lambda_1}\le
\frac{|\Omega|^{2/3}\|h\|^2_{L^2}}{C_{BLY}\nu^2},
$$
we finally get
\begin{multline*}
q(N)\le-\frac{3\nu}{10}|\Omega|^{-2(1-s)/3}
\frac{C_{BLY}G_1^s}{1+C_{BLY}^sG_1^s}N^{(5-2s)/3}+\\+\frac34
\frac{|\Omega|^{2/3}\|h\|^2_{L^2}}{\nu^3}
C_{Sob,6}^3
|\Omega|^{\frac23(s-\frac12)}C_{BLY}^{\frac32-s}N^{(4-2s)/3}=\\=
\frac{3\nu}{10}|\Omega|^{-2(1-s)/3}
\frac{C_{BLY}G_1^s}{1+C_{BLY}^sG_1^s}N^{(4-2s)/3}\(-N^{1/3}+
\frac52C_{Sob,6}^3C_{BLY}^{-3/2-s}\frac{1+C_{BLY}^sG_1^s}{G_1^s}G^2\)
\end{multline*}
and the first estimate of the theorem is proved.

\color{black}

\medskip

Let now $s\in[\frac34,1]$. Then we can use the Cwikel-Lieb-Rosenblum type
inequality \eqref{A.L2-big} for estimating the traces of $\Cal
L_{u(t)}$ and get
\color{black}(see also \cite{IlKoZe} for factor $\sqrt{2/3}$ below)\color{black}
\begin{multline*}
\Tr_N\Cal L_{u(t)}\le-\nu\hat\zeta(\alpha,s,N)-\sqrt{\frac23}\|\Nx u(t)\|_{L^2}\|\rho\|_{L^2}\le\\\le
-\nu\hat\zeta(\alpha,s,N)+\sqrt{\frac23}
C_{s,CLR}|\Omega|^{\frac{2s}3-\frac12}N^{1-\frac{2s}3}\|\Nx u(t)\|_{L^2}.
\end{multline*}
After that, averaging in time and  using (with $\alpha=1$) the obvious estimate
\begin{equation}\label{obvious}
\color{black}
\frac1t\int_0^t\|\Nx u(s)\|_{L^2}^\alpha\,ds\le
\(\frac1t\int_0^t\|\Nx u(s)\|^2_{L^2}\,ds\)^{\alpha/2}\!\!,
\quad 1\le\alpha\le2
\color{black}
\end{equation}
we get
\begin{multline*}
q(N)\le -\nu\hat \zeta(\alpha,s,N)+
\sqrt{\frac23}
C_{s,CLR}|\Omega|^{\color{black}\frac{2s}3-\frac16\color{black}}
C_{BLY}^{-\frac12}\nu^{-1}\|h\|_{L^2}N^{1-\frac{2s}3}=
\\=-\nu\frac3{5}|\Omega|^{\frac{2(s-1)}3}\!
\frac{C_{BLY}^{}G_1^s}{1+C_{BLY}^sG_1^s}N^{\frac53-\frac{2s}3}+
\sqrt{\frac23}
C_{s,CLR}|\Omega|^{\frac{2s}3-\frac16}
C_{BLY}^{-\frac12}\nu^{-1}\|h\|_{L^2}N^{1-\frac{2s}3}=\\=
-\nu\frac3{5}|\Omega|^{\frac{2(s-1)}3}
\frac{C_{BLY}^{}G_1^s}{1+C_{BLY}^sG_1^s}N^{1-\frac{2s}3}\!\!
\(N^{2/3}-\frac{5}3\sqrt{\frac23}
C_{s,CLR}C_{BLY}^{-\frac32} \frac{1+C_{BLY}^sG_1^s}{G_1^s}G\)
\end{multline*}
and the second estimate of the theorem is proved.
\par
Let now $s\in[\frac12,\frac34]$. Then we use \eqref{A.L2-small} to get
$$
\Tr_N\Cal L_{u(t)}\le-\nu R+\sqrt{\frac23}\|\Nx u(t)\|_{L^2}
C_{s,CLR}^{\frac1{4(1-s)}}
\color{black}C_{Sob,6}^\frac{3-4s}{2(1-s)}\color{black}N^{\frac{3-2s}{12(1-s)}}
R^{\frac{3-4s}{4(1-s)}},
$$
where $R:=\sum_{k=1}^N\|\Nx\theta_k\|^2_{L^2}$. Using the Young inequality
$$
BR^{\frac{3-4s}{4(1-s)}}\le\frac{\nu(3-4s)}{4(1-s)}R+\frac1{4(1-s)}\nu^{4s-3}B^{4(1-s)},
$$
together with the inequality $R\ge\hat\zeta(\alpha,s,N)$,
we end up with
$$
\Tr_N\Cal L_{u(t)}\le -\frac\nu{4(1-s)}\hat\zeta(\alpha,s,N)+\frac{2^{2(1-s)}}{4(1-s)3^{2(1-s)}}
C_{s,CLR}\color{black}C_{Sob,6}^{6-8s}\color{black}\nu^{4s-3}N^{1-\frac{2s}3}\|\Nx u(t)\|_{L^2}^{4(1-s)}.
$$
 Using \eqref{A.gest1}, averaging in time and using
\eqref{obvious} with $\alpha=4(1-s)$, we arrive at
\begin{multline*}
q(N)\le -\frac{\nu}{4(1-s)}\frac3{\color{black}5\color{black}}|\Omega|^{\frac{2(s-1)}3}
\frac{C_{BLY}^{}G_1^s}{1+C_{BLY}^sG_1^s}N^{\frac53-\frac{2s}3}+\\+
\frac{2^{2(1-s)}}{4(1-s)3^{2(1-s)}}
C_{s,CLR}\color{black}C_{Sob,6}^{6-8s}\color{black}\,\nu^{8s-7}C_{BLY}^{2(s-1)}|\Omega|^{\frac{4(1-s)}3}
N^{1-\frac{2s}3}\|h\|_{L^2}^{4(1-s)}=\\=
\frac{\nu}{4(1-s)}\frac3{5}|\Omega|^{\frac{2(s-1)}3}
\frac{C_{BLY}^{}G_1^s}{1+C_{BLY}^sG_1^s}N^{1-\frac{2s}3}
\times\\\times\(-N^{2/3}+\frac{5}3\frac{2^{2(1-s)}}{3^{2(1-s)}}
C_{Sob,6}^{6-8s}C_{s,CLR}C_{BLY}^{\color{black}2s-3\color{black}}\frac{1+C_{BLY}^sG_1^s}{G_1^s}G^{4(1-s)}\)
\end{multline*}
which gives the third estimate of the theorem and finishes its proof.
\end{proof}
We see that estimate \eqref{3.est1} is essentially worse than \eqref{3.est2} and \eqref{3.est-mid}. The purpose to present \eqref{3.est1} is to illustrate what can be obtained without using the advanced spectral inequalities and based only on the elementary corollaries of min-max principle. In turn, this emphasize the power of spectral inequalities in estimating the dimension of attractors in hydrodynamics.

 Our second main theorem studies the case where $G_1\ll1$. For simplicity, we consider below only the case $s=1$ although similar results can be obtained for $s<1$ as well.
 \begin{theorem}\label{Th3.main2} Let $s=1$ and
\begin{equation}\label{condition}\color{black}
G_1^{3}G^{4}\le\(\frac9{20}\)^2 C_{BLY}^{-\frac13}C_{LT}^{-1}C_{Sob,6}^{-2}C_{1.CLR}^{-\frac43}.
 \color{black}
 \end{equation}
 Then, the following estimate holds:
 \begin{equation}\label{3.est3}
 \dim_f(\Cal A, V^1)\le \biggl(\color{black}\frac{20}9C_{BLY}^{-2}C_{LT}^{1/2}\color{black}C_{Sob,6}C_{1,CLR}^{\frac23}\biggr)^{\frac9{13}}
 G_1^{-\frac6{13}}G^{\frac{18}{13}}.
 \end{equation}
 \end{theorem}
\begin{proof} In this case, it is more convenient to avoid scaling time and endow the phase space $V^1$ with the standard metric
\begin{equation}\label{non-scaled}
(u_1,u_2)_{\bar V^1_\alpha}:=(u_1,u_2)+\alpha(Au_1,u_2)=((1+\alpha A)u_1,u_2).
\end{equation}
Then, the generator of the linearized equation is
$$
\overline{\Cal L}_{u(t)}v:-(1+\alpha A)^{-1}\(\nu Av-\Pi((v,\Nx)u(t)+(u(t),\Nx)v)\)
$$
and, therefore,
$$
(\overline{\Cal L}_{u(t)}\theta,\theta)_{\bar V^1_\alpha}=-\nu\|\nabla\theta\|^2_{L^2}-((\theta,\Nx)u(t),\theta).
$$
Thus, we basically need to estimate the same quantities, but now the system $\{\theta_k\}_{k=1}^N$ should be orthonormal with respect to the non-scaled inner product \eqref{non-scaled}. Therefore,
$$
\Tr_N\overline{\Cal L}_{u(t)}\le-\nu R+\sqrt{2/3}\|\Nx u(t)\|_{L^2}\|\rho\|_{L^2}.
$$
The key problem here is to estimate the $L^2$-norm of $\rho$ properly. Indeed, using Lieb-Thirring inequality \eqref{A.LT2} and the fact that the $\theta_k$'s are suborthonormal in $L^2$, we have
\begin{equation}
\|\rho\|_{L^2}\le C_{LT}^{3/8}C_{Sob,6}^{3/4}R^{3/4}.
\end{equation}
On the other hand, using 
inequality \eqref{A.L2-big} and the fact that the system $\{\alpha^{1/2}\theta_k\}$ is suborthonormal in $V^1$, we have
\begin{equation}
\|\rho\|_{L^2(\Omega)}\le C_{1,CLR}|\Omega|^{1/6}\alpha^{-1}N^{1/3}.
\end{equation}
Combining two last estimates, we arrive at
$$
\|\rho\|_{L^2}=\|\rho\|_{L^2}^{2/3}\|\rho\|_{L^2}^{1/3}\le
C_{LT}^{1/4}C_{Sob,6}^{1/2}C_{1,CLR}^{1/3}
|\Omega|^{1/{18}}\alpha^{-1/3}N^{1/9}R^{1/2}
$$
and, consequently,
\begin{multline*}
\Tr_N\overline{\Cal L}_{u(t)}\le-\nu R+C_{LT}^{1/4}C_{Sob,6}^{1/2}C_{1,CLR}^{1/3}
|\Omega|^{1/{18}}\alpha^{-1/3}N^{1/9}R^{1/2}\sqrt{2/3}\|\Nx u(t)\|_{L^2}\le\\\le-\frac\nu2R+\frac{\|\Nx u(t)\|^2_{L^2}}{3\nu}C_{LT}^{1/2}C_{Sob,6}C_{1,CLR}^{2/3}
|\Omega|^{1/9}\alpha^{-2/3}N^{2/9}.
\end{multline*}
Assume now that $N$ satisfies $C_{BLY}G_1N^{2/3}\le1$. Then using \eqref{A.ggest}, averaging in time and recalling that $\alpha=G_1|\Omega|^{2/3}$ we have
\begin{multline*}
q(N)\le -\frac{3\nu}{\color{black}20\color{black}}C_{BLY}|\Omega|^{-\frac23}N^{\frac53}+
\frac{\|h\|^2_{L^2}}{2\nu^3C_{BLY}}C_{LT}^{\frac12}C_{Sob,6}C_{1,CLR}^{\frac23}
|\Omega|^{\frac13}G_1^{-\frac23}N^{\frac29}=\\= \frac{3\nu}{20}C_{BLY}|\Omega|^{-\frac23}N^{\frac29}
\(-N^{\frac{13}9}+\color{black}\frac{20}9\color{black}C_{BLY}^{-2}C_{LT}^{\frac12}C_{Sob,6}C_{1,CLR}^{\frac23}G_1^{-\frac23}G^2\)
\end{multline*}
and this gives estimate \eqref{3.est3} if the smallness condition $C_{BLY}G_1N^{\frac23}\le1$ is satisfied for $N$s, which \color{black} gives  the upper bound of $\dim_f(\Cal A)$ of \eqref{3.est3}. \color{black}The smallness condition  is guaranteed by \eqref{condition}. Indeed, use the right hand side of \eqref{3.est3} as $N$  in this condition, we have
$$
C_{BLY}G_1\biggl(\color{black}\frac{20}9C_{BLY}^{-2}C_{LT}^{1/2}
C_{Sob,6}C_{1,CLR}^{\frac23}\biggr)^{\frac6{13}}
G_1^{-\frac4{13}}G^{\frac{12}{13}}\le 1
$$
and this is equivalent to \eqref{condition}\color{black}. Thus, the theorem is proved.
\end{proof}
\begin{remark} All of the estimates obtained above can be written and justified for $h\in H^{-1}_\sigma(\Omega)$. We write the $L^2$-norm in order to keep the traditional definition of the Grashof number only.
\end{remark}

\appendix
\section{Some useful estimates}
In this appendix,  we give some estimates for the sums related with
orthonormal and suborthonormal systems which are the key
ingredients for estimation of the attractor dimension.
We recall that the system of vectors $\{\theta_i\}_{i=1}^N$ is suborthonormal
in a Hilbert space $H$ if
\begin{equation}\label{A.sub}
\sum_{i,j=1}^N\xi_i\xi_j(\theta_i,\theta_j)=\|\sum_{i}\xi_i\theta_i\|^2_H\le |\xi|^2,\ \ \forall \xi\in\R^N.
\end{equation}
We will use this concept in the following situation.
\begin{lemma}\label{LemA1} Let $H_1$ and $H_2$ be two Hilbert spaces
and let \color{black}$\{\theta_i\}_{i=1}^N$
 be suborthonormal in $H:=H_1\cap H_2$ \color{black} with respect to the  inner product generated by the following norm:
$$
\|u\|^2_H:=\alpha_1^2\|u\|^2_{H_1}+\alpha_2^2\|u\|^2_{H_2},\  \ \alpha_1,\alpha_2>0.
$$
Then the systems of vectors $\{\alpha_1\theta_i\}_{i=1}^N$
and $\{\alpha_2\theta_i\}_{i=1}^N$ are suborthonormal in $H_1$ and $H_2$ respectively.
\par
\color{black} Moreover, if  $\{\theta_i\}_{i=1}^N$ is suborthonormal in a Hilbert space $Z$ and $L: Z\to Y$ is a linear map from $Z\to Y$, where $Y$ is another Hilbert space and the norm of $L$ does not exceed one, then the system $\{L\theta_i\}_{i=1}^N$ is suborthonormal in $Y$.
\color{black}
\end{lemma}
\begin{proof}Indeed,
\begin{multline*}
\sum_{i,j}\alpha_1^2\xi_i\xi_j(\theta_i,\theta_j)_{H_1}\le
\biggl\|\sum_i\alpha^2_2\xi_i\theta_i\biggr\|^2_{H_2}+
\sum_{i,j}\alpha_1^2\xi_i\xi_j(\theta_i,\theta_j)_{H_1}=\\=
\sum_{i,j}\xi_i\xi_j\bigl(\alpha_1^2(\theta_i,\theta_j)_{H_1}+
\alpha_2^2(\theta_i,\theta_j)_{H_2}\bigr)=\sum_{i,j}\xi_i\xi_j\delta_{ij}=|\xi|^2,
\end{multline*}
and the proof for $\alpha_2\theta_i$ is analogous. \color{black} For the second part of the lemma we have
$$
\sum_{i,j}\xi_i\xi_j(L\theta_i,L\theta_j)=\|L(\sum_i\xi_i\theta_i)\|^2_Y\le\|\sum_i\xi_i\theta_i\|^2_Z\le|\xi|^2
$$
and the lemma is proved. \color{black}
\end{proof}
We start with the version of the Lieb--Thirring inequality,
 proved in  \cite{Tem} for orthonormal systems for bounded
 domains with the constant depending on $\Omega$.
\begin{proposition} Let $\{\theta_i\}_{i=1}^N\subset H^1(\R^3)$
be suborthonormal in $L^2(\R^3)$ and let
$$
\rho(x):=\sum_{i=1}^N\theta_i^2(x).
$$
Then
\begin{equation}\label{A.LT2}
\|\rho\|_{L^2}\le C_{LT}^{\frac38}C_{Sob,6}^{\frac34}\(\sum_{i=1}^N\|\Nx\theta_i\|^2_{L^2}\)^{\frac34},
\end{equation}
where $C_{LT}$ is the Lieb-Thirring constant (for the $L^{\frac53}$-norm of $\rho$) and $C_{Sob,6}$ is the embedding constant for $H^1_0\subset L^6$.
\end{proposition}
\begin{proof} Indeed,
the classical Lieb--Thirring inequality says that
$$
\|\rho\|_{L^{\frac53}}^{\frac53}\le C_{LT}\sum_{i=1}^N\|\Nx\theta_i\|^2_{L^2},
$$
where $C_{LT}$ is \color{black} an absolute constant, see \cite{FLW}
for its  best to date estimate\color{black}.
 On the other hand, 
\begin{equation}\label{rhoL3}
\|\rho\|_{L^3}=\|\sum_i\theta_i^2\|_{L^3}\le\sum_i\|\theta_i\|_{L^6}^2\le
 C_{Sob,6}^2\sum_{i=1}^N\|\Nx\theta_i\|^2_{L^2}.
\end{equation}
\color{black}Finally, the H\"older inequality\color{black},
namely,
$$
\|\rho\|_{L^2}\le\|\rho\|_{L^{\frac53}}^{\frac58}\|\rho\|^{\frac38}_{L^3}
$$
 gives the desired estimate and finishes the proof of the proposition.
\end{proof}
We now turn to the estimates for the analogue of the spectral $\zeta$ finction for the Stokes operator in a bounded domain $\Omega$. Let $0<\lambda_1\le\lambda_2\le\cdots$ be the eigenvalues of this Stokes operator with Dirichlet boundary conditions, $s\in[0,1]$ and let
$$
\zeta(\alpha,s,N):=\sum_{i=1}^N\frac{\lambda_i}{1+(\alpha\lambda_i)^s}.
$$
Using the Berezin--Li--Yau inequality \color{black}
(see \cite{FLW} for the classical case of the Dirichlet
Laplacian and \cite{FA09} for the case of the Stokes operator)
\color{black}
\begin{equation}\label{BLY}
\lambda_k\ge\frac{\lambda_1+\cdots+\lambda_k}k\ge C_{BLY}|\Omega|^{-\frac23} k^{\frac23},
\end{equation}
together with the monotonicity of the function $\lambda\to \frac{\lambda}{1+(\alpha\lambda)^s}$, we get
\begin{equation}\label{A.zeta}
\zeta(\alpha,s,N)\ge
\color{black}C_{BLY}^{}\color{black}|\Omega|^{-\frac23}\sum_{k=1}^N
\frac{k^{\frac23}}{1+C_{BLY}^sG_1^sk^{\frac{2s}3}},
\end{equation}
where $G_1:=\alpha|\Omega|^{-\frac23}$ is the second dimensionless
\color{black} parameter\color{black}. 
This allows us to get the following proposition.
\begin{proposition}\label{PropA.est} Let $s\in[0,1]$ and $\alpha>0$. Then,
\begin{equation}\label{A.gest}
\zeta(\alpha,s,N)\ge \color{black}\frac 3{5}C_{BLY}^{}\color{black}
|\Omega|^{-2/3}\frac1{1+C_{BLY}^s G_1^s}N^{(5-2s)/3}.
\end{equation}
Moreover, if $C_{BLY}G_1N^{2/3}\le1$, we have
\begin{equation}\label{A.ggest}
\zeta(\alpha,s,N)\ge \color{black}\frac3{10}C_{BLY}^{}\color{black}
|\Omega|^{-2/3}N^{5/3}.
\end{equation}
\end{proposition}
\begin{proof} Let $\beta:=C_{LBY}^{s}G_1^s$. Then
\begin{multline*}
S=\sum_{k=1}^N\frac{k^{\frac23}}{1+\beta k^{\frac{2s}3}}=
\frac1\beta\sum_{k=1}^N
\frac{k^{\frac{2(1-s)}3}(\beta k^{\frac{2s}3}+1)-k^{\frac{2(1-s)}3}}{1+\beta k^{\frac{2s}3}}=\\
=\frac1\beta\sum_{k=1}^Nk^{\frac{2(1-s)}3}-\frac1\beta \sum_{k=1}^N
\frac {k^{\frac{2(1-s)}3}}{1+\beta k^{\frac{2s}3}}\ge \frac1\beta
\sum_{k=1}^Nk^{\frac{2(1-s)}3} -\frac1\beta S,
\end{multline*}
\color{black} which gives \eqref{A.gest}\color{black}:
$$
S\ge\frac1{1+\beta}\sum_{k=1}^Nk^{\frac{2(1-s)}3}\ge \frac1{1+\beta}\frac3{5-2s}
\color{black}N\color{black}^{\frac3{5-2s}}\ge
\color{black} \frac3{5}\color{black}\frac1{1+\beta}N^{\frac{5-2s}3}.
$$
\par
In the case when $C_{BLY}G_1N^{2/3}\le1$, we have that $1+\beta k^{2s/3}\le2$ for all $k\le N$ and this gives estimate \eqref{A.ggest} and finishes the proof of the proposition.
\end{proof}
Let us now consider the function
\begin{equation}\label{A.gest1}
\hat\zeta(\alpha,s,N):=\sum_{k=1}^N\frac{\lambda_k}{\alpha^{-s}+\lambda_k^s}=
\alpha^s\zeta(\alpha,s,N)\ge
\frac3{\color{black}5\color{black}}|\Omega|^{-\frac{2(1-s)}3}
\frac{\color{black}C_{BLY}^{}\color{black}G_1^s}{1+C_{BLY}^sG_1^s}N^{\frac{5-2s}3},
\end{equation}
where we have used \eqref{A.gest}.
\par
Let now $H^s_A:=D(A^{s/2})$, where $A^s$ is a spectral fractional
 power of the Stokes operator with Dirichlet boundary conditions
in $\Omega$ with the norm $\|u\|_{H^s_A}=\|A^{s/2}u\|_{L^2}$
and let $\{\theta_k\}_{k=1}^N$ be a suborthonormal system in $H^s_A$.
Then, since
$$
D(A^s)=D((-\Delta)^s)\cap\{\divv u=0,\ \ u\cdot n\big|_{\partial\Omega}=0\}, \ \ 0\le s\le 1,
$$
where $(-\Dx)^s$ is a spectral fractional Laplacian in $\Omega$, $H^s_A$
is a closed subspace of $D((-\Dx)^{s/2})$ and, in particular, the norms
$\|A^{s/2}u\|_{L^2}$ and $\|(-\Dx)^{s/2}u\|_{L^2}$ are equivalent on $H_A^s$.
\par
Furthermore, let $\ext_0(u)$ be the  extention of a function
$u(x)$, $x\in\Omega$ by zero when $x\notin\Omega$ and let $(-\Dx)^s_{\R^3}$
be the fractional Laplacian on $\R^3$ (defined, e.g., by the Fourier transform).
Then, an alternative choice of the fractional Laplacian in $\Omega$,
which is sometimes refferred as ``regional fractional Laplacian",
is the following
$$
(-\Dx)^s_{\R^3,\Omega}u:=(-\Dx)^s_{\R^3}(\ext_0(u)).
$$
It is also known, see e.g., \cite{BSV}, that $D((-\Dx)^{s})=D((-\Dx)^{s}_{\R^3,\Omega})$, so the corresponding norms are equivalent on $H^s_A$:
$$
\|A^{s/2}u\|_{L^2(\Omega)}\sim\|(-\Dx)^{s/2}u\|_{L^2(\Omega)}\sim\|(-\Dx)_{\R^3}^{s/2}(\ext_0(u))\|_{L^2(\R^3)}
$$
and, therefore, there exists a dimensionless constant $C_s(\Omega)$, such that
\begin{equation}\label{A.bad}
\|(-\Dx)^{s/2}_{\R^3}(\ext_0(u))\|_{L^2(\R^3)}\le C_s(\Omega)\|u\|_{H^s_A}
\ \  \forall u\in H^s_A.
\end{equation}
\color{black}
\begin{lemma}\label{LemA.sub} The constants $C_s(\Omega)\le1$ for $s\in[0,1]$ and, therefore, the system of vectors $\{\ext_0(\theta_i)\}_{i=1}^N$ (where $\{\theta_i\}_{i=1}^N$ is suborthonormal in $H^s_A$) is suborthonormal in $\dot H^s(\R^3)$.
\end{lemma}
\begin{proof} Let us start with comparison of Stokes and Laplace operators. We know that, for any $u\in H^1_A=H^1_0(\Omega)\cap\{\divv u=0\}\subset H^1_{-\Dx}=H^1_0(\Omega)$, we have
\begin{multline*}
\|u\|^2_{H^1_A}=\|A^{1/2}u\|^2_{L^2}=(Au,u)=\\=(-\Pi\Dx u,u)=(-\Dx u,u)=\|\Nx u\|^2_{L^2}=\|(-\Dx)^{1/2}u\|^2_{L^2}=\|u\|^2_{H^1_{-\Delta}}.
\end{multline*}
Thus, the identity operator $I$ (the embedding operator) acts from $H^1_A$ to $H^1_{-\Dx}$ and also from $H^0_A$ to $H_{-\Dx}^0=L^2(\Omega)$ and its norm in both cases equals to one. On the other hand, by the Parseval equality,
$$
\|u\|^2_{H^s_A}=\sum_{i=1}^\infty\lambda_i^s(u,e_i)^2,\ \ \|u\|^2_{H^s_{-\Dx}}=\sum_{i=1}^\infty\mu_i^s(u,\hat e_i)^2,
$$
where $\{\lambda_i,e_i\}$ and $\{\mu_i,\hat e_i\}$ are the eigenvalues and eigenfunctions for the Stokes and Laplace operators respectively. Therefore, the spaces $H^s_A$ and $H^s_{-\Dx}$ are isometric to the weighted spaces of sequences  $l^2_{\boldsymbol{\lambda}^s}$ and $l^2_{\boldsymbol{\mu}^s}$. Since for weighted $l^2$-spaces complex interpolation also gives isometries between the corresponding weighted and interpolation norms, we conclude that the identity operator acts from $H^s_A$ to $H^s_{-\Dx}$ and its norm does not exceed one:
$$
\|u\|_{H^s_{-\Dx}}\le \|u\|_{H^s_A},\ s\in[0,1],\ u\in H^s_A.
$$
Analogously, by the Fourier transform and Plancherel equality, the norm in the homogeneous Sobolev space $\dot H^s(\R^3)$ is
$$
\|u\|^2_{\dot H^s}=((-\Dx)^s_{\R^3}u,u)=\|(-\Dx)^{s/2}_{\R^3} u\|^2_{L^2}=C\int_{R^3}|\xi|^{2s}|\hat u(\xi)|^2\,d\xi,
$$
where $\hat u$ is the Fourier transform of $u$ and $C=C(d)$ is an absolute constant depending only on the dimension $d$. By this reason, the extension operator $\ext_0$ actually acts from the interpolation pair $[H^1_0(\Omega),L^2(\Omega)]=[l^2_{\boldsymbol{\mu}^1}, l^2]$ to the pair $[\dot H^1,L^2]=[L^2_{|\xi|^2},L^2]$ and its norms at the endpoints equal to one. Using the complex interpolation again, we end up with the estimate
$$
\|\ext_0(u)\|_{\dot H^s(\R^3)}\le \|u\|_{H^s_{-\Dx}}\le \|u\|_{H^s_A},\ s\in[0,1],\ \ u\in H^s_A
$$
and together with Lemma \ref{LemA1} finish the proof of the lemma.
\end{proof}
\color{black}

We combine this lemma with the inequality
\color{black}
for families of functions $\{\theta_k\}_{k=1}^N$ which are  (sub)orthonormal in
$\dot H^s(\R^3)$, see \cite{Lieb}
\begin{equation}\label{Lieb}
\|\rho\|_{L^{\frac3{3-2s}}(\R^3)}\le C_{s,CLR}N^{1-\frac{2s}3}
\end{equation}
for $0\le s\le 1$ (the notation for the constant is explained at the end of this
work).
\color{black}
Applying this inequality to our system of
vectors which are suborthonormal in $H^s_A$, we arrive at
\begin{equation}
\label{CLR}
\|\rho\|_{L^{\frac3{3-2s}}(\Omega)}\le C_s(\Omega)^2C_{s,CLR}N^{1-\frac{2s}3}\le
C_{s,CLR}N^{1-\frac{2s}3}.
\end{equation}
Finally, we use the H\"older inequality  to get the desired estimates for the $L^2$-norm of~$\rho$.
\begin{proposition}\label{PropA.l2} Let $\{\theta_k\}_{k=1}^N$ be a system of suborthonormal vectors in $H^s_A$, Then, for $0\le s\le\frac34$, we have
\begin{equation}\label{A.L2-small}
\|\rho\|_{L^2(\Omega)}\le \color{black}C_{Sob,6}^{\frac{3-4s}{2(1-s)}}\color{black} 
C_{s,CLR}^{\frac1{4(1-s)}}N^{\frac{3-2s}{12(1-s)}}
\(\sum_{k=1}^N\|\Nx\theta_k\|^2_{L^2}\)^{\frac{3-4s}{4(1-s)}}.
\end{equation}
For the case $3/4\le s\le 1$, we have another estimate
\begin{equation}\label{A.L2-big}
\|\rho\|_{L^2(\Omega)}\le 
C_{s,CLR}|\Omega|^{\frac{2s}3-\frac12}N^{1-\frac{2s}3}.
\end{equation}
\end{proposition}
\begin{proof} Indeed, for the first estimate we use
H\"older's inequality
$$
\|\rho\|_{L^2}\le\|\rho\|_{L^{\frac3{3-2s}}}^\kappa\|\rho\|_{L^3}^{1-\kappa}
$$
with $\kappa=\frac1{4(1-s)}$  and $\frac3{3-2s}\le2$
combined with \eqref{CLR} and the already obtained estimate
\eqref{rhoL3} for the $L^3$
norm of $\rho$.
For the case $s\ge3/4$, we use H\"older inequality
$$
\|\rho\|_{L^2}\le |\Omega|^{\frac{2s}3-\frac12}\|\rho\|_{L^{\frac3{3-2s}}}
$$
and \eqref{CLR} gives us the desired estimate and finishes the proof of the proposition.
\end{proof}
We conclude with one more useful inequality related with min-max principle.
\begin{proposition}\label{PropA.mm} Let $\frac12\le s\le1$, $\eb\ge0$, let $\{\bar\theta_k\}_{k=1}^N$ be (sub)orthonormal in $L^2$ and let
$$
\theta_k:=(\eb+A^s)^{-1/2}\bar\theta_k.
$$
Then the following estimate holds:
\begin{equation}\label{A.mm}
\sum_{k=1}^N\|\theta_k\|^2_{L^3}\le \frac3{\color{black}2\color{black}}C_{Sob,6}|\Omega|^{\frac23(s-\frac12)}C_{BLY}^{\frac12-s}N^{\frac23(2-s)},
\end{equation}
where $C_{Sob,3}$ is the embedding constant $H^{1/2}_A\subset L^3$ and
by complex interpolation it holds that
\color{black}$C_{Sob,3}\le C_{Sob,6}^{1/2}$.
\end{proposition}
\begin{proof}Using the embedding $H^{1/2}_A\subset L^3$, we get
$$
\|\theta_k\|^2_{L^3}\le C_{Sob,3}^2\|A^{1/4}(\eb+A^s)^{-1/2}\bar\theta_k\|^2_{L^2}=C^2_{Sob,3}
(A^{1/2}(\eb+A^s)^{-1}\bar\theta_k,\bar\theta_k)
$$
Then, by min max principle,
$$
\sum_{k=1}^N\|\theta_k\|^2_{L^3}\le C_{Sob,3}^2\sum_{k=1}^N\lambda_k^{\frac12-s}\le \frac3{\color{black}2\color{black}}
C_{Sob,3}^2C_{BLY}^{\frac12-s}|\Omega|^{\frac23(s-\frac12)}N^{\frac23(\frac12-s)+1}.
$$
\end{proof}

We will now give explicit values of the dimensionless constants
that played crucial role in the dimension estimates.

1. The sharp constant in the Sobolev inequality
$$
\|u\|_{L^{2d/(d-2)}}\le C_{Sob}(d)\|\nabla u\|_{L^2}
$$
is well known, see, for instance,  \cite{FLW} for a rearrangement-free proof:
$$
C_{Sob,6}:=C_{Sob}(3)=\frac1{\sqrt{3}}\(\frac2\pi\)^{2/3}.
$$
We also point out that in the vector case (in our consideration)
the constant is preserved.

\medskip

2. The Berezin--Li--Yau constant (sharp, by notational definition) for the Stokes operator
in inequality \eqref{BLY}
was found in \cite{FA09}:
$$
C_{BLY}(d)=\frac2{2+d}\left(\frac{(2\pi)^d}{\omega_d(d-1)}\right)^{2/d},
$$
where $\omega_d$ denotes the volume of the unit ball in $\mathbb R^d$,
so that
$$
C_{BLY}=\frac25(3\pi^2)^{2/3}.
$$

\medskip

3. For the Lieb--Thirring constant we have
  \cite{FLW}
$$
C_{LT}\le
    \frac56\frac{2^{1/3}}{\pi^{4/3}}\cdot(1.456\dots)^{2/3}\,.
$$
In the scalar case, the constant $C_{LT}(d)$
is related to the constant $\mathrm L_{1,d}$ by the equality~\cite{FLW}
$$
\mathrm{c}_{\mathrm {LT}}(d)=(2/d)(1+d/2)^{1+2/d}\mathrm L_{1,d}^{2/d},
$$
where $\mathrm L_{1,d}$ is the constant in the Lieb--Thirring bound
for the negative trace of the Schrodinger operator
$-\Delta-V(x)$, $V(x)\ge0$ in $\mathbb R^d$:
$$
\Tr(-\Delta-V)_-\le\mathrm L_{1,d}\int_{\mathbb R^d}V(x)^{1+d/2}dx.
$$
In turn, the best  to date estimate of the constant
$\mathrm L_{1,d}$ is (see \cite{FLW})
$\mathrm L_{1,d}\le\mathrm L_{1,d}^\mathrm{cl}\cdot 1.456\dots$, where
$\mathrm L_{1,d}^\mathrm{cl}=\frac{\omega_d}{(2\pi)^d}\frac2{d+2}$.
Moreover, in the vector case, the constant $\mathrm L_{1,d}$
goes over to $d\mathrm L_{1,d}$. In the two-dimensional divergence-free
case, the constant does not double.

The fact that in going over  to a suborthonormal
system the constant in the Lieb--Thirring inequality does not
increase is shown in \cite{I2005}.

\medskip

4. Finally, turning to inequality \eqref{Lieb} we point out that for $s=1$
the constant $C_{1,CLR}$ is the  particular case ($d=3$) in the  inequality
due to E.\,Lieb \cite{Lieb}
$$
\|\rho\|_{L^p}\le L_{0,d}^{2/d}\frac d{d-2}N^{(d-2)/d},\quad p=d/(d-2).
$$
Straightforward passage to the vector-valued case and the case of suborthonormal systems
was done in \cite{IKZ}:
$$
\|\rho\|_{L^p}\le (dL_{0,d})^{2/d}\frac d{d-2}N^{(d-2)/d},\quad p=d/(d-2).
$$
Here $L_{0,d}$ is the constant in the celebrated Cwikel--Lieb--Rozenblum
bound for the number $N(0,-\Delta-V)$ of  negative eigenvalues of
the Schr\"odinger operator $-\Delta-V$, $V(x)\ge0$ in $\mathbb
R^d$, see~\cite{Cwikel,L,R}:
\begin{equation*}\label{CLRbound}
N(0,-\Delta-V)\le L_{0,d}\int_{\mathbb R^d}V(x)^{d/2}dx.
\end{equation*}
The best to date bound for $L_{0,3}$ is Lieb's bound~\cite{L}
$$L_{0,3}\le 6.8693\cdot L_{0,3}^{\mathrm{cl}}=0.116\dots.
$$
Hence,
$$
C_{1,CLR}\le3(3\cdot 0.116\dots)^{2/3}=1.484251229.
$$

In the fractional case $s\in[1/2,1]$ arguing as in \cite{Lieb} (see also~\cite{IKZ})
we have
$$
\|\rho\|_{L^{d/{(d-2s)}}}\le (dL_{0,d}(s))^{2s/d}\frac d{d-2s}N^{(d-2s)/d},
$$
where  $L_{0,d}(s)$ is the constant in the fractional Cwikel--Lieb--Rozenblum
bound for the number  of  negative eigenvalues of
the Schr\"odinger operator $(-\Delta)^s-V$, $V(x)\ge0$ in $\mathbb
R^d$:
\begin{equation*}\label{CLRbounds}
N(0,(-\Delta)^s-V)\le L_{0,d}(s)\int_{\mathbb R^d}V(x)^{d/2s}dx,
\quad s<d/2.
\end{equation*}
The explicit estimate of the constant $L_{0,d}(s)$ was found in
\cite{FrankJST}:
$$
L_{0,d}(s)\le
\left(\frac{d(d+2s)}{(d-2s)^2}\right)^{(d-2s)/2s}\frac{\omega_d}{(2\pi)^d}\frac d{d-2s},
$$
therefore
$$
C_{s,CLR}\le
\left[\left(\left(3\frac{3(3+2s)}{(3-2s)^2}\right)^{(3-2s)/2s}
\frac{1}{6\pi^2}\frac 3{3-2s}\right)^{2s/3}\frac3{3-2s}\right],
$$
or, maximising the corresponding expression on the interval
$s\in[1/2,1]$, we find that on the interval $s\in[1/2,1]$
$$
C_{s,CLR}\le\max_{s\in[1/2,1]}[\dots]=
[\dots]_{s=1/2}=\frac94\frac{6^{1/3}}{\pi^{2/3}}=1.906044430.
$$
{\bf Acknowledgement.} This work was supported by Moscow Center of
Fundamental and Applied Mathematics, Agreement with the Ministry
of Science and Higher Education of the Russian Federation, No. 075-15-2025-346 (AI) as well by Russian Science Foundation grant no.
25-11-20069 (SZ).

\end{document}